\documentclass[11pt]{amsart}

\usepackage[a4paper,margin=2.7cm]{geometry}
\usepackage[T1]{fontenc}
\usepackage{lmodern}
\usepackage{microtype}
\usepackage{amsmath,amssymb,amsthm,amsfonts,mathtools}
\usepackage{enumitem}
\usepackage{mathrsfs}
\usepackage{hyperref}
\hypersetup{hidelinks}

\numberwithin{equation}{section}

\newtheorem{theorem}{Theorem}[section]
\newtheorem{lemma}[theorem]{Lemma}
\newtheorem{proposition}[theorem]{Proposition}
\newtheorem{corollary}[theorem]{Corollary}
\newtheorem{question}[theorem]{Question}

\theoremstyle{definition}
\newtheorem{definition}[theorem]{Definition}

\theoremstyle{remark}
\newtheorem{remark}[theorem]{Remark}
\newtheorem{example}[theorem]{Example}

\newcommand{\R}{\mathbb R}
\newcommand{\C}{\mathbb C}
\newcommand{\K}{\mathbb K}
\newcommand{\N}{\mathbb N}
\newcommand{\Id}{\mathrm{Id}}
\newcommand{\wstar}{w^*}
\newcommand{\Exp}{\mathsf E}
\newcommand{\Prob}{\mathsf P}

\DeclareMathOperator{\spanop}{span}

\title[Finite-codimensional subspaces of Daugavet spaces]
{Finite-codimensional subspaces of Daugavet spaces:\\
projection constants and minimal projections}

\author[T.~Kania]{Tomasz Kania}
\address[T.~Kania]{Mathematical Institute, Czech Academy of Sciences, \v{Z}itn\'a 25, 115 67 Praha 1, Czech Republic\newline
and Institute of Mathematics and Computer Science, Jagiellonian University, {\L}ojasiewicza 6, 30-348 Krak\'ow, Poland}
\email{kania@math.cas.cz, tomasz.marcin.kania@gmail.com}
\thanks{IM CAS (RVO 67985840).}

\author[G.~Lewicki]{Grzegorz Lewicki}
\address[G.~Lewicki]{Institute of Mathematics and Computer Science, Jagiellonian University, {\L}ojasiewicza 6, 30-348 Krak\'ow, Poland}
\email{Grzegorz.Lewicki@im.uj.edu.pl}


\begin{document}

\begin{abstract}
Over the real or complex field, we establish a duality formula for projection constants of finite-codimensional subspaces of Banach spaces with the Daugavet property. If
\[
Y=\bigcap_{j=1}^n \ker f_j \subset X,
\qquad
W=\spanop\{f_1,\dots,f_n\} \subset X^*,
\]
then
\[
\lambda(Y,X)=1+\lambda(W,X^*),
\]
and minimal projections onto $Y$ correspond exactly to weak$^*$-continuous minimal projections onto $W$. This yields, in particular, a complete description of the hyperplane case: every hyperplane has projection constant $2$, and $\ker f$ admits a minimal projection if and only if $f$ attains its norm.

We then specialise to the real space $X=C[0,1]$. Our second ingredient is a transfer principle from duplication-stable finite-dimensional subspaces of $\ell_1^N$ to piecewise-constant subspaces of $L_1[0,1]\subset M[0,1]=C[0,1]^*$. For the regular symmetric spaces constructed by Chalmers and the second-named author and the second named author and Prophet, respectively, the transferred subspaces retain their projection constants but admit no weak$^*$-continuous minimal projections. Passing to annihilators yields finite-codimensional subspaces of the real space $C[0,1]$ for which the infimum defining the projection constant is not attained.

As a consequence, for every $\Lambda\in[2,\infty)$ there exists a finite-codimensional subspace $Y$ of the real space $C[0,1]$ such that
\[
\lambda(Y,C[0,1])=\Lambda,
\]
and the infimum defining $\lambda(Y,C[0,1])$ is not attained. For each even codimension $n$ we moreover realise every value in the interval $(2,1+\beta_n]$, where
\[
\beta_n
=
\Exp_{\Prob_n}\Bigl|\sum_{j=1}^n \varepsilon_j\Bigr|
=
n2^{-n}\binom{n}{n/2}
\sim
\sqrt{\frac{2n}{\pi}},
\]
$(\varepsilon_j)$ is a Rademacher family on $\Omega_n=\{-1,1\}^n$, and $\Prob_n$ is the uniform probability measure.
\end{abstract}

\maketitle

\section{Introduction}

For a closed subspace $E$ of a Banach space $F$, the \emph{projection constant} of $E$ in $F$ is
\[
\lambda(E,F):=\inf\{\|P\|: P\colon F\to E,\ P|_E=\Id_E\}.
\]
When the infimum is attained, one obtains a \emph{minimal projection}; when it is not, the pair $(E,F)$ exhibits a substantially subtler approximation-theoretic geometry. Projection constants and minimal projections form a classical subject, with roots in the work of Bohnenblust, Sobczyk, Kelley, Grünbaum, Kadec--Snobar, and many others. They lie at the intersection of approximation theory, local Banach space geometry, and the structure of function spaces. 

Minimal projections and projection constants have been studied extensively in approximation theory and Banach space geometry. Classical contributions include the work of Chalmers and Metcalf \cite{CM1,CM2}, Cheney and collaborators \cite{CF,CH,CM}, and Fisher--Morris--Wulbert \cite{FMW}. Later developments on uniqueness, asymptotic behaviour, and extremal projection constants may be found in \cite{HK1,HK3,HK4,KLL,LE2,LS,SK3,LM,FS2,BB22,BB23}. The regular symmetric finite-dimensional models that enter our construction come from \cite{CL1,CL2,CL3,CL4}, while the uniqueness input we need is closely related to the Chalmers--Metcalf framework and its later refinement in \cite{LP}.

The present paper is concerned with finite-codimensional subspaces of Banach spaces with the Daugavet property. This class contains, for instance, $C(K)$ whenever $K$ has no isolated points, as well as $L_1(\mu)$ over non-atomic measure spaces; see \cite{KSSW}. In such spaces, finite-rank perturbations of the identity satisfy the rigid norm identity
\[
\|\Id_X+T\|=1+\|T\|,
\]
and this turns the projection problem for a finite-codimensional subspace into a dual finite-dimensional problem.

More precisely, if
\[
Y=\bigcap_{j=1}^n \ker f_j\subset X
\qquad\text{and}\qquad
W=\spanop\{f_1,\dots,f_n\}\subset X^*,
\]
then our first main result shows that
\[
\lambda(Y,X)=1+\lambda(W,X^*),
\]
and that minimal projections $X\to Y$ correspond exactly to weak$^*$-continuous minimal projections $X^*\to W$. Thus the geometry of finite-codimensional subspaces in a Daugavet space is encoded by the finite-dimensional annihilator in the dual, together with the additional weak$^*$-continuity constraint.

This perspective already yields several structural consequences. Every proper finite-codimensional subspace of a Daugavet space has projection constant at least $2$, every hyperplane has projection constant exactly $2$, and a hyperplane $\ker f$ admits a minimal projection if and only if the functional $f$ attains its norm. In particular, the hyperplane case is completely classified in terms of norm-attainment.

The second part of the paper specialises to the real space $X=C[0,1]$, or equivalently to its dual $M[0,1]$. Here our aim is not merely to estimate projection constants, but to produce systematic non-attainment phenomena. The starting point is a transfer principle from finite-dimensional subspaces of $\ell_1^N$ to piecewise-constant subspaces of $L_1[0,1]\subset M[0,1]$. If $V\subset \ell_1^N$ is \emph{duplication-stable}, in the sense that all its block duplicates have the same projection constant and unique minimal projections, then its piecewise-constant copy $\mathscr V\subset M[0,1]$ has the same projection constant as $V$ but admits no weak$^*$-continuous minimal projection as soon as $\dim V\geqslant 2$.

The core mechanism is simple and robust. The unique minimal projection in $L_1[0,1]$ is forced, by uniqueness on all uniform refinements of the underlying partition, to coincide with any putative weak$^*$-continuous minimal projection on $M[0,1]$. Weak$^*$ continuity then expresses the latter by continuous coding functions. A localisation argument inside each interval of the partition shows that these coding functions must be constant, which is incompatible with dimension at least two; in the complex case one first rotates by a unimodular scalar and compares real parts. Through the Daugavet duality, this obstruction transfers immediately to annihilators in the real space $C[0,1]$.

To make this effective we import explicit duplication-stable spaces from the theory of regular symmetric subspaces of $\ell_1$ \cite{CL1,CL2,CL3,CL4}, together with the later work of the second-named and Prophet \cite{LP}. A complementary recent construction of the authors \cite{KL26}, in the different setting of $(\lambda^+)$-injective spaces, shows that zero-sum subspaces can amplify relative projection constants while preserving non-attainment. This yields finite-codimensional subspaces of the real space $C[0,1]$ with prescribed projection constants and no minimal projections. In codimension $2$ we obtain every value in a non-trivial interval, and for every even $n$ we obtain every value in the interval
\[
(2,1+\beta_n],
\qquad
\beta_n
=
\Exp_{\Prob_n}\Bigl|\sum_{j=1}^n \varepsilon_j\Bigr|
=
n2^{-n}\binom{n}{n/2},
\]
where $(\varepsilon_j)_{j=1}^n$ is a Rademacher family on the uniform probability space $\Omega_n=\{-1,1\}^n$. Equivalently, $\beta_n$ is the mean absolute displacement after $n$ steps of the simple symmetric random walk, so
\[
\beta_n\sim \sqrt{\frac{2n}{\pi}}.
\]
Since every codimension-$n$ subspace of a Daugavet space satisfies $\lambda\leqslant 1+\sqrt n$ by the Kadec--Snobar theorem, our explicit family is asymptotically optimal up to the universal factor $\sqrt{2/\pi}$.

Our main results are the following.

\begin{theorem}
\label{thm:intro-duality}
Let $X$ be a Banach space with the Daugavet property, and let
\[
Y=\bigcap_{j=1}^n \ker f_j\subset X,
\qquad
W=\spanop\{f_1,\dots,f_n\}\subset X^*.
\]
Then
\[
\lambda(Y,X)=1+\lambda(W,X^*).
\]
Moreover, a minimal projection from $X$ onto $Y$ exists if and only if there exists a weak$^*$-continuous minimal projection from $X^*$ onto $W$.
\end{theorem}

\begin{theorem}
\label{thm:intro-main}
For every $\Lambda\in[2,\infty)$ there exists a finite-codimensional subspace $Y$ of the real space $C[0,1]$ such that
\[
\lambda(Y,C[0,1])=\Lambda,
\]
and the infimum defining $\lambda(Y,C[0,1])$ is not attained.
\end{theorem}

We also obtain explicit codimension-$2$ and codimension-$n$ families; see Theorems~\ref{thm:codim2} and \ref{thm:even-codim}. At the endpoint $\Lambda=1+\sqrt n$, Theorem~\ref{thm:intro-duality} would require an annihilator $W\subset M[0,1]$ that is extremal for the Kadec--Snobar bound. Thus the remaining gap between $1+\beta_n$ and $1+\sqrt n$ sits precisely at the interface between Daugavet geometry, extremal projection-constant theory, and weak$^*$ non-attainment.

The paper is organised as follows. Sections~\ref{sec:prelim}--\ref{sec:transfer} are written over both scalar fields. Section~\ref{sec:prelim} contains the weak$^*$ approximation lemma for finite-dimensional ranges. Section~\ref{sec:daugavet} proves the Daugavet duality formula and derives the hyperplane criterion and general codimension bounds. Section~\ref{sec:transfer} develops the passage from $\ell_1^N$ to piecewise-constant subspaces of $L_1[0,1]\subset M[0,1]$ and shows the failure of weak$^*$ continuity for minimal projections. Section~\ref{sec:examples} then specialises to the real field, isolates the precise external input from \cite{CL1, LP}, and derives explicit non-attaining examples in the real space $C[0,1]$. The final section collects further remarks and questions.

\section{Preliminaries}
\label{sec:prelim}

In Sections~\ref{sec:prelim}--\ref{sec:transfer} we allow both scalar fields, $\K\in\{\R,\C\}$. Section~\ref{sec:examples} will be specialised to the real field. For a Banach space $X$, the duality pairing between $X$ and $X^*$ will be written as $\langle x^*,x\rangle=x^*(x)$.

If $W\subset X^*$ is finite-dimensional, we define
\[
\lambda_{\wstar}(W,X^*)
:=
\inf\{\|Q\|: Q\colon X^*\to W\text{ is a weak$^*$-continuous projection}\}.
\]
Clearly
\[
\lambda(W,X^*)\leqslant \lambda_{\wstar}(W,X^*).
\]
The converse inequality is the only place where we use the principle of local reflexivity.

\begin{proposition}
\label{prop:wstar-equals-lambda}
Let $X$ be a Banach space and let $W\subset X^*$ be finite-dimensional. Then
\[
\lambda_{\wstar}(W,X^*)=\lambda(W,X^*).
\]
\end{proposition}

\begin{proof}
Fix $\varepsilon>0$ and choose a projection $Q\colon X^*\to W$ such that
\[
\|Q\|<\lambda(W,X^*)+\varepsilon.
\]
Consider the finite-dimensional subspace
\[
E:=Q^*(W^*)\subset X^{**}.
\]
By the principle of local reflexivity, applied to the pair $(E,W)$ (see, for instance, \cite[Theorem~11.2.4]{AK}), there exists a linear operator
\[
T\colon E\to X
\]
such that
\[
\|T\|\leqslant 1+\varepsilon
\qquad\text{and}\qquad
\langle f,Te\rangle = \langle e,f\rangle
\quad (e\in E,\ f\in W).
\]
Define
\[
\widetilde Q := (T\circ Q^*)^*\colon X^*\to W^{**}=W.
\]
Since $T\circ Q^*$ takes values in $X$, the operator $\widetilde Q$ is weak$^*$-continuous. We claim that it is a projection onto $W$. Indeed, let $f\in W$ and $\psi\in W^*$. Then
\[
\langle \psi,\widetilde Q f\rangle
= \langle TQ^*\psi,f\rangle
= \langle Q^*\psi,f\rangle
= \langle \psi,Qf\rangle
= \langle \psi,f\rangle,
\]
so $\widetilde Qf=f$ for all $f\in W$.

Finally,
\[
\|\widetilde Q\| = \|TQ^*\| \leqslant \|T\|\,\|Q^*\| \leqslant (1+\varepsilon)\|Q\|
< (1+\varepsilon)(\lambda(W,X^*)+\varepsilon).
\]
Since $\varepsilon>0$ is arbitrary, this shows that
\[
\lambda_{\wstar}(W,X^*)\leqslant \lambda(W,X^*).
\]
The reverse inequality is immediate.
\end{proof}

We shall also use the following standard description of finite-rank weak$^*$-continuous operators on $M(K)=C(K)^*$.

\begin{lemma}
\label{lem:wstar-coded}
Let $K$ be a compact Hausdorff space, let $M(K)=C(K)^*$, and let $E\subset M(K)$ be finite-dimensional. A finite-rank operator $P\colon M(K)\to E$ is weak$^*$-continuous if and only if there exist a basis $\mu_1,\dots,\mu_n$ of $E$ and functions $g_1,\dots,g_n\in C(K)$ such that
\[
P\nu = \sum_{j=1}^n \Bigl(\int_K g_j\, d\nu\Bigr)\mu_j
\qquad (\nu\in M(K)).
\]
\end{lemma}

\begin{proof}
If $P$ has the displayed form, then it is obviously weak$^*$-continuous. Conversely, choose a basis $\mu_1,\dots,\mu_n$ of $E$ and let $\phi_1,\dots,\phi_n\in E^*$ be the coordinate functionals. If $P$ is weak$^*$-continuous, then each $\phi_j\circ P$ is a weak$^*$-continuous linear functional on $M(K)$, hence belongs to the canonical copy of $C(K)$ inside $M(K)^*$. Thus there exists $g_j\in C(K)$ such that
\[
(\phi_j\circ P)(\nu)=\int_K g_j\,d\nu
\qquad (\nu\in M(K)).
\]
Therefore
\[
P\nu = \sum_{j=1}^n (\phi_j\circ P)(\nu)\,\mu_j
= \sum_{j=1}^n \Bigl(\int_K g_j\, d\nu\Bigr)\mu_j.
\qedhere
\]
\end{proof}

\section{Finite-codimensional subspaces of Daugavet spaces}
\label{sec:daugavet}

Recall that a Banach space $X$ has the \emph{Daugavet property} if every rank-one operator $T\colon X\to X$ satisfies
\[
\|\Id_X+T\| = 1+\|T\|.
\]
A theorem of Kadets, Shvidkoy, Sirotkin, and Werner implies that then the same identity holds for every finite-rank operator; see \cite{KSSW}.

\begin{theorem}
\label{thm:daugavet-duality}
Let $X$ be a Banach space with the Daugavet property, let $f_1,\dots,f_n\in X^*$ be linearly independent, and put
\[
Y:=\bigcap_{j=1}^n \ker f_j,
\qquad
W:=\spanop\{f_1,\dots,f_n\}.
\]
Then
\[
\lambda(Y,X)=1+\lambda(W,X^*).
\]
Moreover, the following are equivalent.
\begin{enumerate}[label=\textup{(\roman*)}]
\item There exists a minimal projection $P\colon X\to Y$.
\item There exists a weak$^*$-continuous minimal projection $Q\colon X^*\to W$.
\end{enumerate}
If $P$ and $Q$ correspond, then
\[
Q=(\Id_X-P)^*.
\]
\end{theorem}

\begin{proof}
Let $P\colon X\to Y$ be a projection. Since $f_1,\dots,f_n$ are linearly independent and $Y=\bigcap_{j=1}^n\ker f_j$, the quotient map
\[
x+Y\longmapsto \bigl(\langle f_1,x\rangle,\dots,\langle f_n,x\rangle\bigr)
\]
is an isomorphism from $X/Y$ onto $\K^n$. Choose $x_1,\dots,x_n\in X$ such that
\[
\langle f_i,x_j\rangle=\delta_{ij}
\qquad (1\leqslant i,j\leqslant n),
\]
and set
\[
\varphi_j:=(\Id_X-P)x_j \in \ker P
\qquad (1\leqslant j\leqslant n).
\]
Then
\[
\langle f_i,\varphi_j\rangle
=
\langle f_i,x_j\rangle-\langle f_i,Px_j\rangle
=
\delta_{ij},
\]
because $Px_j\in Y$. For $x\in X$, define
\[
z:=(\Id_X-P)x-\sum_{j=1}^n \langle f_j,x\rangle\varphi_j.
\]
Since both terms lie in $\ker P$, we have $z\in\ker P$. On the other hand, for each $i$,
\[
\langle f_i,z\rangle
=
\langle f_i,x-Px\rangle
-
\sum_{j=1}^n \langle f_j,x\rangle\langle f_i,\varphi_j\rangle
=
\langle f_i,x\rangle-\langle f_i,x\rangle
=
0,
\]
so $z\in Y$. Hence $z\in Y\cap\ker P=\{0\}$, and therefore
\[
Px=x-\sum_{j=1}^n \langle f_j,x\rangle\varphi_j
\qquad (x\in X).
\]
Consequently,
\[
(\Id_X-P)^*g = \sum_{j=1}^n \langle g,\varphi_j\rangle f_j
\qquad (g\in X^*),
\]
so $(\Id_X-P)^*$ is a weak$^*$-continuous projection from $X^*$ onto $W$. Since $\Id_X-P$ has finite rank and $X$ has the Daugavet property,
\[
\|P\| = \|\Id_X-(\Id_X-P)\| = 1+\|\Id_X-P\| = 1+\|(\Id_X-P)^*\|.
\]
Taking infima over all projections $P\colon X\to Y$ gives
\[
\lambda(Y,X)\geqslant 1+\lambda_{\wstar}(W,X^*).
\]
By Proposition~\ref{prop:wstar-equals-lambda},
\[
\lambda(Y,X)\geqslant 1+\lambda(W,X^*).
\]

Conversely, let $Q\colon X^*\to W$ be a weak$^*$-continuous projection. Choose the basis $f_1,\dots,f_n$ of $W$. By weak$^*$ continuity there exist $\varphi_1,\dots,\varphi_n\in X$ such that
\[
Qg = \sum_{j=1}^n \langle g,\varphi_j\rangle f_j
\qquad (g\in X^*).
\]
Because $Qf_i=f_i$, we obtain
\[
\langle f_i,\varphi_j\rangle = \delta_{ij}
\qquad (1\leqslant i,j\leqslant n).
\]
Define
\[
Px := x - \sum_{j=1}^n \langle f_j,x\rangle\varphi_j
\qquad (x\in X).
\]
Then, for each $i$,
\[
\langle f_i,Px\rangle
=
\langle f_i,x\rangle
-
\sum_{j=1}^n \langle f_j,x\rangle\langle f_i,\varphi_j\rangle
=
0,
\]
so $P(X)\subset Y$. If $y\in Y$, then $\langle f_j,y\rangle=0$ for all $j$, hence $Py=y$. Thus $P$ is a projection from $X$ onto $Y$, and $(\Id_X-P)^*=Q$. Therefore, again by the Daugavet property,
\[
\|P\| = 1+\|Q\|.
\]
Taking infima over all weak$^*$-continuous projections $Q\colon X^*\to W$ yields
\[
\lambda(Y,X)\leqslant 1+\lambda_{\wstar}(W,X^*) = 1+\lambda(W,X^*).
\]
This proves the equality.

Now suppose that $P$ is minimal. Then the associated operator $Q=(\Id_X-P)^*$ is weak$^*$-continuous and
\[
\|Q\| = \|P\|-1 = \lambda(Y,X)-1 = \lambda(W,X^*),
\]
so $Q$ is a weak$^*$-continuous minimal projection onto $W$.

Conversely, if $Q$ is a weak$^*$-continuous minimal projection onto $W$, the associated projection $P$ onto $Y$ satisfies
\[
\|P\| = 1+\|Q\| = 1+\lambda(W,X^*) = \lambda(Y,X),
\]
so $P$ is minimal.
\end{proof}

\begin{corollary}
\label{thm:daugavet-bounds}
Let $X$ be a Banach space with the Daugavet property, and let $Y\subset X$ be a proper closed subspace of finite codimension $n$. Then
\[
2\leqslant \lambda(Y,X)\leqslant 1+\sqrt n.
\]
In particular, every hyperplane in $X$ has projection constant $2$.
\end{corollary}

\begin{proof}
Let $W=Y^\perp\subset X^*$. Then $\dim W=n$, so Theorem~\ref{thm:daugavet-duality} gives
\[
\lambda(Y,X)=1+\lambda(W,X^*).
\]
Because $W\neq\{0\}$, one has $\lambda(W,X^*)\geqslant 1$, which yields the lower bound. The upper bound follows from the Kadec--Snobar theorem applied to the $n$-dimensional space $W\subset X^*$; see, for instance, \cite{AK}. If $n=1$, then the upper and lower bounds coincide.
\end{proof}

\begin{corollary}
\label{cor:hyperplane}
Let $X$ be a Banach space with the Daugavet property, and let $f\in X^*\setminus\{0\}$. Then
\[
\lambda(\ker f,X)=2.
\]
Moreover, the following are equivalent.
\begin{enumerate}[label=\textup{(\roman*)}]
\item There exists a minimal projection from $X$ onto $\ker f$.
\item The functional $f$ attains its norm on the unit sphere of $X$.
\end{enumerate}
\end{corollary}

\begin{proof}
Replacing $f$ by $f/\|f\|$, we may assume that $\|f\|=1$. Put $Y=\ker f$ and $W=\spanop\{f\}$. Since $W$ is one-dimensional, there exists $\Phi\in S_{X^{**}}$ such that $\Phi(f)=1$, and then
\[
Qg:=\Phi(g)f
\qquad (g\in X^*)
\]
is a projection from $X^*$ onto $W$ of norm $1$. Hence $\lambda(W,X^*)=1$, and Theorem~\ref{thm:daugavet-duality} yields
\[
\lambda(Y,X)=1+\lambda(W,X^*)=2.
\]

By Theorem~\ref{thm:daugavet-duality}, a minimal projection onto $Y$ exists if and only if there is a weak$^*$-continuous minimal projection $Q\colon X^*\to W$. Any weak$^*$-continuous projection onto $W$ has the form
\[
Qg=g(x)f
\qquad (g\in X^*)
\]
for some $x\in X$ with $f(x)=1$. Since $\|Q\|=\|x\|\,\|f\|=\|x\|$, such a projection is minimal if and only if $\|x\|=1$. This is exactly the norm-attainment of $f$.
\end{proof}

\begin{example}
\label{ex:hyperplanes}
The hyperplane criterion already yields concrete non-attaining examples.
\begin{enumerate}[label=\textup{(\alph*)}]
\item Let
\[
\mu:=\frac12\Bigl(\sum_{m=1}^\infty 2^{-m}\delta_{1/m}-\delta_0\Bigr)\in M[0,1].
\]
Then $\|\mu\|=1$, but $\mu$ does not attain its norm on $C[0,1]$. Indeed, suppose that $h\in C[0,1]$ satisfies $\|h\|_\infty\leqslant 1$ and
\[
\Bigl|\int h\,d\mu\Bigr|=1.
\]
Choose $\omega\in\K$ with $|\omega|=1$ and $\int \omega h\,d\mu=1$. Then
\[
1
=
\frac12\left|\sum_{m=1}^\infty 2^{-m}(\omega h)(1/m)-(\omega h)(0)\right|
\leqslant
\frac12\left(\sum_{m=1}^\infty 2^{-m}|h(1/m)|+|h(0)|\right)
\leqslant 1.
\]
Hence equality holds throughout, forcing $(\omega h)(1/m)=1$ for all $m$ and $(\omega h)(0)=-1$, contradicting continuity. Therefore $\ker\mu\subset C[0,1]$ is a hyperplane with projection constant $2$ and no minimal projection.
\item Let $X=L_1[0,1]$ and let $f\in L_\infty[0,1]=X^*$ be given by $f(t)=t$. Then $\|f\|_\infty=1$, but $f$ does not attain its norm on $L_1[0,1]$, because for every non-zero $g\in L_1[0,1]$,
\[
\Bigl|\int_0^1 t g(t)\,dt\Bigr|\leqslant \int_0^1 t|g(t)|\,dt < \int_0^1 |g(t)|\,dt=\|g\|_1.
\]
Hence $\ker f\subset L_1[0,1]$ is again a hyperplane with projection constant $2$ and no minimal projection.
\end{enumerate}
\end{example}

\begin{corollary}
\label{cor:C01-annihilator}
Let $W\subset M[0,1]=C[0,1]^*$ be finite-dimensional, and let
\[
Y:=\{f\in C[0,1]: \langle \mu,f\rangle =0\text{ for all }\mu\in W\}.
\]
Then
\[
\lambda(Y,C[0,1])=1+\lambda(W,M[0,1]).
\]
Moreover, if $W$ admits no weak$^*$-continuous minimal projection in $M[0,1]$, then the infimum defining $\lambda(Y,C[0,1])$ is not attained.
\end{corollary}

\begin{proof}
The space $C[0,1]$ has the Daugavet property, so the first assertion follows from Theorem~\ref{thm:daugavet-duality}. The second assertion is immediate from the equivalence of minimal projections established there.
\end{proof}

\section{A transfer principle for piecewise-constant copies}
\label{sec:transfer}

We write $M[0,1]=C[0,1]^*$ for the space of regular Borel measures on $[0,1]$, and we identify $L_1[0,1]$ with the subspace of absolutely continuous measures.

For $N\in\N$, let
\[
I_j:=\Bigl[\frac{j-1}{N},\frac{j}{N}\Bigr)
\quad (1\leqslant j\leqslant N-1),
\qquad
I_N:=\Bigl[\frac{N-1}{N},1\Bigr].
\]
Set
\[
\nu_j := N\chi_{I_j}\in L_1[0,1]
\qquad (1\leqslant j\leqslant N).
\]
Then $\|\nu_j\|_{L_1}=1$, and the map
\[
J_N\colon \ell_1^N\to Z_N:=\spanop\{\nu_1,\dots,\nu_N\}\subset L_1[0,1],
\qquad
J_N(e_j)=\nu_j,
\]
is a surjective linear isometry.

Associated to this partition are contractive projections
\[
A_N\colon L_1[0,1]\to Z_N,
\qquad
A_Nf := \sum_{j=1}^N \Bigl(\int_{I_j} f(t)\,dt\Bigr)\nu_j,
\]
and
\[
R_N\colon M[0,1]\to Z_N,
\qquad
R_N\mu := \sum_{j=1}^N \mu(I_j)\,\nu_j.
\]
Indeed,
\[
\|R_N\mu\|_{M} = \sum_{j=1}^N |\mu(I_j)| \leqslant \|\mu\|_M.
\]

\begin{definition}
Let $V\subset \ell_1^N$ be finite-dimensional. The subspace
\[
\mathscr V := J_N(V) \subset Z_N \subset L_1[0,1] \subset M[0,1]
\]
will be called the \emph{piecewise-constant copy} of $V$.
\end{definition}

To formulate the uniqueness statement cleanly, we isolate the stability under coordinate duplication.

\begin{definition}
Let $V\subset \ell_1^N$ be finite-dimensional. For $m\in\N$, define the duplication operator
\[
D_m\colon \ell_1^N\to \ell_1^{Nm},
\qquad
D_m(x_1,\dots,x_N)
=
\Bigl(\underbrace{\tfrac{x_1}{m},\dots,\tfrac{x_1}{m}}_{m},\dots,
\underbrace{\tfrac{x_N}{m},\dots,\tfrac{x_N}{m}}_{m}\Bigr).
\]
We say that $V$ is \emph{duplication-stable} if, for every $m\in\N$, the space $D_m(V)$ has the same relative projection constant as $V$ and admits a unique minimal projection.
\end{definition}

For $m\in\N$, subdivide each $I_j$ into $m$ equal subintervals
\[
I_{j,r}^{(m)}:=\Bigl[\frac{j-1}{N}+\frac{r-1}{Nm},\frac{j-1}{N}+\frac{r}{Nm}\Bigr)
\qquad (1\leqslant j\leqslant N,\ 1\leqslant r\leqslant m),
\]
with the obvious adjustment at the endpoint $1$, and set
\[
\nu_{j,r}^{(m)}:=Nm\,\chi_{I_{j,r}^{(m)}}.
\]
Let
\[
Z_{N,m}:=\spanop\{\nu_{j,r}^{(m)}:1\leqslant j\leqslant N,\ 1\leqslant r\leqslant m\}\subset L_1[0,1].
\]
The map
\[
J_{N,m}\colon \ell_1^{Nm}\to Z_{N,m},
\qquad
J_{N,m}(e_{j,r})=\nu_{j,r}^{(m)},
\]
is a surjective isometry.

\begin{lemma}
\label{lem:ZNm-dense}
The union $\bigcup_{m\in\N} Z_{N,m}$ is dense in $L_1[0,1]$.
\end{lemma}

\begin{proof}
Since $C[0,1]$ is dense in $L_1[0,1]$, it suffices to approximate continuous functions. Fix $h\in C[0,1]$. For $m\in\N$, let $h_m$ be the function which is constant on each $I_{j,r}^{(m)}$ and equal there to the average of $h$ over $I_{j,r}^{(m)}$. Then $h_m\in Z_{N,m}$, and the uniform continuity of $h$ implies
\[
\|h-h_m\|_\infty \to 0
\qquad (m\to\infty).
\]
Therefore $\|h-h_m\|_{L_1}\to 0$, and the proof is complete.
\end{proof}

\begin{theorem}
\label{thm:piecewise-main}
Let $V\subset \ell_1^N$ be a duplication-stable subspace of dimension $n\geqslant 2$, and let $\mathscr V\subset M[0,1]$ be its piecewise-constant copy.

\begin{enumerate}[label=\textup{(\roman*)}]
\item
\[
\lambda(\mathscr V,L_1[0,1]) = \lambda(\mathscr V,M[0,1]) = \lambda(V,\ell_1^N).
\]
\item $\mathscr V$ admits a unique minimal projection from $L_1[0,1]$ onto $\mathscr V$.
\item There is no weak$^*$-continuous minimal projection from $M[0,1]$ onto $\mathscr V$.
\end{enumerate}
\end{theorem}

\begin{proof}
Let $P_V\colon \ell_1^N\to V$ be the unique minimal projection. Define
\[
P_0 := J_N P_V J_N^{-1} A_N\colon L_1[0,1]\to \mathscr V.
\]
Then $P_0$ is a projection onto $\mathscr V$.

Since $A_N$ is contractive and $J_N$ is an isometry,
\[
\|P_0\|\leqslant \|P_V\| = \lambda(V,\ell_1^N).
\]
On the other hand, the restriction of $P_0$ to $Z_N$ equals $J_NP_VJ_N^{-1}$, so
\[
\|P_0\|\geqslant \|J_NP_VJ_N^{-1}\| = \lambda(V,\ell_1^N).
\]
Hence
\[
\|P_0\| = \lambda(V,\ell_1^N).
\]
Because $Z_N\subset L_1[0,1]$, we have
\[
\lambda(\mathscr V,L_1[0,1]) \geqslant \lambda(\mathscr V,Z_N)=\lambda(V,\ell_1^N),
\]
so in fact
\[
\lambda(\mathscr V,L_1[0,1]) = \lambda(V,\ell_1^N).
\]

Now define
\[
\widetilde P_0 := J_N P_V J_N^{-1} R_N\colon M[0,1]\to \mathscr V.
\]
Exactly the same argument yields
\[
\lambda(\mathscr V,M[0,1]) = \lambda(V,\ell_1^N),
\]
proving \textup{(i)}.

We prove uniqueness in $L_1[0,1]$. Fix $m\in\N$. Under the isometry $J_{N,m}$, the subspace $\mathscr V$ corresponds to the duplicated space $D_m(V)\subset \ell_1^{Nm}$.

Let $B_m\colon \ell_1^{Nm}\to D_m(\ell_1^N)$ be the block-averaging projection,
\[
B_m(z_{j,r})_{j,r} = \Bigl(\frac1m\sum_{r=1}^m z_{1,r},\dots,\frac1m\sum_{r=1}^m z_{1,r},\dots,
\frac1m\sum_{r=1}^m z_{N,r},\dots,\frac1m\sum_{r=1}^m z_{N,r}\Bigr).
\]
Then the restriction $P_0|_{Z_{N,m}}$ corresponds, via $J_{N,m}$, to the operator
\[
D_m P_V D_m^{-1} B_m\colon \ell_1^{Nm}\to D_m(V),
\]
which is minimal. Since $V$ is duplication-stable, this minimal projection onto $D_m(V)$ is unique.

Now let $P\colon L_1[0,1]\to \mathscr V$ be any minimal projection. If $P\neq P_0$, choose $f\in L_1[0,1]$ such that $Pf\neq P_0f$. By Lemma~\ref{lem:ZNm-dense}, there exist $m$ and $h\in Z_{N,m}$ with $Ph\neq P_0h$. Because $V$ is duplication-stable,
\[
\lambda(\mathscr V,Z_{N,m}) = \lambda(D_m(V),\ell_1^{Nm}) = \lambda(V,\ell_1^N)=\lambda(\mathscr V,L_1[0,1]).
\]
Hence both restrictions $P|_{Z_{N,m}}$ and $P_0|_{Z_{N,m}}$ have norm equal to $\lambda(\mathscr V,Z_{N,m})$; that is, both are minimal projections from $Z_{N,m}$ onto $\mathscr V$. Via $J_{N,m}$ they correspond to minimal projections from $\ell_1^{Nm}$ onto $D_m(V)$. By uniqueness, they must coincide, a contradiction. Therefore $P=P_0$, proving \textup{(ii)}.

Finally, suppose that $P\colon M[0,1]\to \mathscr V$ is a weak$^*$-continuous minimal projection. By Lemma~\ref{lem:wstar-coded}, there exist a basis $y_1,\dots,y_n$ of $\mathscr V$ and functions $g_1,\dots,g_n\in C[0,1]$ such that
\[
P\mu = \sum_{k=1}^n \Bigl(\int_0^1 g_k\, d\mu\Bigr) y_k
\qquad (\mu\in M[0,1]).
\]
Since $\|P|_{L_1}\|\leqslant \|P\| = \lambda(\mathscr V,M[0,1]) = \lambda(\mathscr V,L_1[0,1])$, the restriction $P|_{L_1}$ is a minimal projection from $L_1[0,1]$ onto $\mathscr V$. By \textup{(ii)},
\[
P|_{L_1} = P_0.
\]

Fix $k\in\{1,\dots,n\}$ and $j\in\{1,\dots,N\}$. We claim that $g_k$ is constant on $I_j$. If not, choose $s,t\in I_j$ with $g_k(s)\neq g_k(t)$. After interchanging $s$ and $t$ if necessary, we may choose $\omega\in\K$ with $|\omega|=1$ such that
\[
\Re(\omega g_k(s)) < \Re(\omega g_k(t)).
\]
By continuity, there exist disjoint subintervals $J,J'\subset I_j$ of the same length such that
\[
\sup_{u\in J} \Re(\omega g_k(u)) < \inf_{v\in J'} \Re(\omega g_k(v)).
\]
Because $J$ and $J'$ have the same length and lie in the same $I_j$,
\[
A_N(\chi_J)=A_N(\chi_{J'}),
\]
hence
\[
P_0(\chi_J)=P_0(\chi_{J'}).
\]
Since $P|_{L_1}=P_0$ and $y_1,\dots,y_n$ is a basis, this implies
\[
\int_J g_\ell(t)\,dt = \int_{J'} g_\ell(t)\,dt
\qquad (1\leqslant \ell\leqslant n).
\]
In particular,
\[
\Re\left(\omega\int_J g_k(t)\,dt\right)
=
\Re\left(\omega\int_{J'} g_k(t)\,dt\right).
\]
On the other hand,
\[
\Re\left(\omega\int_J g_k(t)\,dt\right)
=
\int_J \Re(\omega g_k(t))\,dt
<
\int_{J'} \Re(\omega g_k(t))\,dt
=
\Re\left(\omega\int_{J'} g_k(t)\,dt\right),
\]
a contradiction. Thus each $g_k$ is constant on every $I_j$, and continuity then forces each $g_k$ to be constant on all of $[0,1]$.

Write $g_k\equiv c_k$. Then for $f\in L_1[0,1]$,
\[
P_0f = \sum_{k=1}^n c_k\Bigl(\int_0^1 f(t)\,dt\Bigr) y_k.
\]
Hence the matrix of $P_0|_{\mathscr V}=\Id_{\mathscr V}$ in the basis $y_1,\dots,y_n$ has the form
\[
\bigl(c_k\!\int_0^1 y_\ell(t)\,dt\bigr)_{k,\ell=1}^n,
\]
which has rank at most $1$. Since $n\geqslant 2$, this matrix cannot be the identity. The contradiction shows that no weak$^*$-continuous minimal projection exists, proving \textup{(iii)}.
\end{proof}

\section{Explicit non-attaining families in the real space \texorpdfstring{$C[0,1]$}{C[0,1]}}
\label{sec:examples}

In this section we work over the real field. We now isolate the precise external input from the theory of regular symmetric subspaces that is needed in the applications.

\begin{proposition}
\label{prop:regular-input}
In the case of real scalars, the following assertions hold.
\begin{enumerate}[label=\textup{(\roman*)}]
\item For every $a>b>0$, there exists a two-dimensional regular symmetric space $V_{a,b}\subset \ell_1^8$ such that
\[
\lambda(V_{a,b},\ell_1^8)=\frac{a^2+ab}{a^2+b^2}.
\]
Moreover, for every $m\in\N$,
\[
\lambda(D_m(V_{a,b}),\ell_1^{8m})=\lambda(V_{a,b},\ell_1^8),
\]
and $D_m(V_{a,b})$ admits a unique minimal projection.
\item For every even $n\geqslant 4$ and every $a\geqslant 0$, there exists an $n$-dimensional regular symmetric space $V_{n,a}\subset \ell_1^{2n!2^n}$ such that
\begin{equation}
\label{eq:rho-na}
\lambda(V_{n,a},\ell_1^{2n!2^n})
= \rho_n(a)
:=
\left(
\frac{2^{n-1}}{C_n+na}
+
\frac{a}{2^{n-1}+a}
\right)^{-1},
\end{equation}
where
\[
C_n:=\sum_{\ell=0}^{n/2-1} \binom n\ell (n-2\ell).
\]
Moreover, for every $m\in\N$,
\[
\lambda(D_m(V_{n,a}),\ell_1^{2mn!2^n})=\lambda(V_{n,a},\ell_1^{2n!2^n}),
\]
and $D_m(V_{n,a})$ admits a unique minimal projection.
\end{enumerate}
\end{proposition}

\begin{proof}
The two-dimensional family and its projection-constant formula are taken from \cite{CL4}. The even-dimensional family and the formula \eqref{eq:rho-na} are taken from \cite{CL3}. For these regular symmetric spaces, duplication preserves the relative projection constant by the block-averaging argument of \cite[Lemma~2.3]{CL3}. For each duplicate $D_m(V)$, where $V\in\{V_{a,b},V_{n,a}\}$, we apply \cite[Theorem~5]{LP} to the minimal extension problem with $A=\Id_{D_m(V)}$; the hypothesis to verify is that the associated Chalmers--Metcalf operator restricts to the identity on $D_m(V)$, and for these duplicated regular symmetric spaces this is exactly the computation encoded in \cite[Lemma~2.3]{CL3}. Hence every duplicate admits a unique minimal projection. This is exactly the external input used below.
\end{proof}

\subsection*{A two-dimensional family}

Let $V_{a,b}\subset \ell_1^8$ be as in Proposition~\ref{prop:regular-input}\textup{(i)}.

\begin{theorem}
\label{thm:codim2}
For every
\[
2<\Lambda\leqslant \frac{3+\sqrt2}{2}
\]
there exists a codimension-$2$ subspace $Y$ of the real space $C[0,1]$ such that
\[
\lambda(Y,C[0,1])=\Lambda,
\]
and the infimum defining $\lambda(Y,C[0,1])$ is not attained.
\end{theorem}

\begin{proof}
Let
\[
\rho(t):=\frac{1+t}{1+t^2}
\qquad (0<t<1).
\]
Then $\rho$ is continuous, and a straightforward calculation shows that it attains its maximum at $t=\sqrt2-1$, with
\[
\max_{0<t<1}\rho(t)=\frac{1+\sqrt2}{2}.
\]
Hence for every
\[
1<r\leqslant \frac{1+\sqrt2}{2}
\]
there exist $a>b>0$ such that
\[
\lambda(V_{a,b},\ell_1^8)=r.
\]
By Proposition~\ref{prop:regular-input}\textup{(i)}, the space $V_{a,b}$ is duplication-stable. Let $\mathscr V_{a,b}\subset M[0,1]$ be the piecewise-constant copy of $V_{a,b}$, and define
\[
Y:=\mathscr V_{a,b}^{\perp}\subset C[0,1].
\]
By Theorem~\ref{thm:piecewise-main},
\[
\lambda(\mathscr V_{a,b},M[0,1]) = r,
\]
and $\mathscr V_{a,b}$ admits no weak$^*$-continuous minimal projection. Therefore Corollary~\ref{cor:C01-annihilator} yields
\[
\lambda(Y,C[0,1]) = 1+r,
\]
and the infimum is not attained. Since $\dim \mathscr V_{a,b}=2$, the space $Y$ has codimension $2$.
\end{proof}

\subsection*{An even-codimensional family}

Fix an even integer $n\geqslant 4$ and $a\geqslant 0$, and let $V_{n,a}\subset \ell_1^{2n!2^n}$ be as in Proposition~\ref{prop:regular-input}\textup{(ii)}.

It is convenient to record the endpoint value in probabilistic form.

\begin{lemma}
\label{lem:Cn-rademacher}
Let $n\geqslant 2$ be even, let
\[
\Omega_n:=\{-1,1\}^n,
\qquad
\Prob_n(A):=2^{-n}\#A
\quad (A\subset \Omega_n),
\]
and let $\varepsilon_j(\omega):=\omega_j$ for $\omega=(\omega_1,\dots,\omega_n)\in\Omega_n$. Equivalently, $\varepsilon_1,\dots,\varepsilon_n$ are independent Rademacher random variables on the product probability space $(\Omega_n,\Prob_n)$. Then
\[
\beta_n
:=
\Exp_{\Prob_n}\Bigl|\sum_{j=1}^n \varepsilon_j\Bigr|
=
\frac{C_n}{2^{n-1}}
=
n2^{-n}\binom{n}{n/2}.
\]
In particular,
\[
\beta_n\sim \sqrt{\frac{2n}{\pi}}
\qquad (n\to\infty,\ n\text{ even}).
\]
\end{lemma}

\begin{proof}
Set
\[
S_n:=\varepsilon_1+\cdots+\varepsilon_n.
\]
The law of $S_n$ is the classical binomial law of the simple symmetric random walk; see, for example, \cite[Chapter~III, Section~2]{Feller}. For completeness we recall the counting argument. For $0\leqslant \ell\leqslant n$, the event $\{S_n=n-2\ell\}$ consists exactly of those sign-vectors in $\Omega_n$ having precisely $\ell$ coordinates equal to $-1$. Therefore
\[
\Prob_n(S_n=n-2\ell)=2^{-n}\binom n\ell
\qquad (0\leqslant \ell\leqslant n).
\]
Using symmetry and the fact that $n$ is even, we obtain
\[
\Exp_{\Prob_n}|S_n|
=
2^{-n}\sum_{\ell=0}^n \binom n\ell |n-2\ell|
=
2^{1-n}\sum_{\ell=0}^{n/2-1} \binom n\ell (n-2\ell)
=
\frac{C_n}{2^{n-1}}.
\]
Moreover,
\[
(n-2\ell)\binom n\ell
=
n\binom{n-1}{\ell}
-
n\binom{n-1}{\ell-1}
\qquad (0\leqslant \ell\leqslant n/2-1),
\]
with the convention $\binom{n-1}{-1}=0$. Hence the sum telescopes:
\[
C_n
=
n\sum_{\ell=0}^{n/2-1}
\left(
\binom{n-1}{\ell}-\binom{n-1}{\ell-1}
\right)
=
n\binom{n-1}{n/2-1}.
\]
Since
\[
\binom{n}{n/2}=2\binom{n-1}{n/2-1},
\]
it follows that
\[
\beta_n=\frac{C_n}{2^{n-1}}=n2^{-n}\binom{n}{n/2}.
\]
Finally, Stirling's formula yields
\[
\binom{n}{n/2}\sim 2^n\sqrt{\frac{2}{\pi n}},
\]
and therefore
\[
\beta_n\sim \sqrt{\frac{2n}{\pi}}.
\qedhere
\]
\end{proof}

\begin{theorem}
\label{thm:even-codim}
Let $n\geqslant 4$ be even, and put
\[
\beta_n:=\frac{C_n}{2^{n-1}}=\Exp_{\Prob_n}\Bigl|\sum_{j=1}^n \varepsilon_j\Bigr|,
\]
with expectation taken with respect to the product probability measure $\Prob_n$ from Lemma~\ref{lem:Cn-rademacher}. Then for every
\[
2<\Lambda\leqslant 1+\beta_n
\]
there exists a codimension-$n$ subspace $Y$ of the real space $C[0,1]$ such that
\[
\lambda(Y,C[0,1])=\Lambda,
\]
and the infimum defining $\lambda(Y,C[0,1])$ is not attained.
\end{theorem}

\begin{proof}
Fix such a $\Lambda$ and put $r:=\Lambda-1\in(1,\beta_n]$. The function $\rho_n$ defined in \eqref{eq:rho-na} is continuous on $[0,\infty)$, satisfies
\[
\rho_n(0)=\frac{C_n}{2^{n-1}}=\beta_n,
\qquad
\lim_{a\to\infty}\rho_n(a)=1,
\]
so there exists $a\geqslant 0$ such that $\rho_n(a)=r$.

By Proposition~\ref{prop:regular-input}\textup{(ii)}, the space $V_{n,a}$ is duplication-stable. Let $\mathscr V_{n,a}\subset M[0,1]$ be the piecewise-constant copy of $V_{n,a}$, and let
\[
Y:=\mathscr V_{n,a}^{\perp}\subset C[0,1].
\]
By Theorem~\ref{thm:piecewise-main},
\[
\lambda(\mathscr V_{n,a},M[0,1]) = r,
\]
and $\mathscr V_{n,a}$ has no weak$^*$-continuous minimal projection. Corollary~\ref{cor:C01-annihilator} now yields
\[
\lambda(Y,C[0,1])=1+r=\Lambda,
\]
with non-attainment. Since $\dim \mathscr V_{n,a}=n$, the codimension of $Y$ equals $n$.
\end{proof}

\begin{corollary}
\label{cor:all-lambda->2}
For every $\Lambda>2$ there exists a finite-codimensional subspace $Y$ of the real space $C[0,1]$ such that
\[
\lambda(Y,C[0,1])=\Lambda,
\]
and the infimum defining $\lambda(Y,C[0,1])$ is not attained.
\end{corollary}

\begin{proof}
By Lemma~\ref{lem:Cn-rademacher}, the numbers $\beta_n$ tend to $\infty$ along the even integers. Choose an even $n\geqslant 4$ such that $\Lambda\leqslant 1+\beta_n$. Then apply Theorem~\ref{thm:even-codim}.
\end{proof}

\begin{corollary}
\label{cor:all-lambda>=2}
For every $\Lambda\in[2,\infty)$ there exists a finite-codimensional subspace $Y$ of the real space $C[0,1]$ such that
\[
\lambda(Y,C[0,1])=\Lambda,
\]
and the infimum defining $\lambda(Y,C[0,1])$ is not attained.
\end{corollary}

\begin{proof}
For $\Lambda>2$, this is Corollary~\ref{cor:all-lambda->2}. For $\Lambda=2$, use Example~\ref{ex:hyperplanes}(a).
\end{proof}

\section{Further remarks and questions}
\label{sec:remarks}

\begin{remark}
For a codimension-$n$ subspace $Y$ of a Daugavet space, Corollary~\ref{thm:daugavet-bounds} gives the general upper bound
\[
\lambda(Y,X)\leqslant 1+\sqrt n.
\]
Our even-codimensional family reaches every value up to
\[
1+\beta_n,
\qquad
\beta_n\sim \sqrt{\frac{2n}{\pi}}.
\]
Thus the explicit non-attaining examples produced here are asymptotically optimal up to the universal factor $\sqrt{2/\pi}$. It would be very interesting to know whether the gap can be closed in $C[0,1]$ by replacing the regular symmetric spaces with less rigid finite-dimensional models.

For $n\in\N$, let
\[
\lambda_n:=\sup\{\lambda(E): \dim E=n\},
\]
where $\lambda(E)$ denotes the absolute projection constant of $E$, that is,
\[
\lambda(E):=\sup\{\lambda(E,Z): E\subset Z \text{ isometrically}\}.
\]
The numbers $\lambda_n$ are classical and notoriously difficult to determine; see \cite{Grunbaum,KLL,BB22,BB23}.
\end{remark}

\begin{question}
Fix $n\in\N$. Does every value in the full interval
\[
(2,1+\lambda_n]
\]
occur as the projection constant of a codimension-$n$ subspace $Y$ of the real space $C[0,1]$ for which no minimal projection exists?
\end{question}

\begin{remark}
If such a space $Y$ existed with
\[
\lambda(Y,C[0,1])=1+\lambda_n,
\]
then Theorem~\ref{thm:daugavet-duality} would force its annihilator
\[
W:=Y^\perp\subset M[0,1]
\]
to satisfy
\[
\lambda(W,M[0,1])=\lambda_n.
\]
Thus $W$ would have to realise the maximal absolute projection constant $\lambda_n$ inside $M[0,1]$. When $\lambda_n<\sqrt n$, this is weaker than extremality for the Kadec--Snobar upper bound, but it is still a very rigid requirement in finite-dimensional projection-constant theory; see, for instance, \cite{Grunbaum,KLL,BB22,BB23}. The endpoint problem therefore asks for an annihilator which is simultaneously extremal for the finite-dimensional projection-constant problem and yet admits no weak$^*$-continuous minimal projection.
\end{remark}

\begin{remark}
The Daugavet duality of Theorem~\ref{thm:daugavet-duality} applies to every Banach space with the Daugavet property. What is specific to $C[0,1]$ is the piecewise-constant transfer in Section~\ref{sec:transfer}. It would be natural to seek analogous constructions in other dual Daugavet spaces, for instance in $M(K)$ for perfect compact spaces $K$, or in duals of $L_1$-spaces, where one might hope to obtain further weak$^*$ non-attainment phenomena for finite-codimensional subspaces. A different amplification mechanism appears in \cite{KL26}, where zero-sum subspaces in finite $\ell_\infty$-sums multiply relative projection constants by explicit factors while preserving non-attainment. It would be interesting to understand whether a Daugavet-space analogue of that construction exists.
\end{remark}

\end{document}